\documentclass[12pt]{amsart}
\usepackage[T1]{fontenc}
\usepackage{graphicx} 
\usepackage{amsfonts,amsthm,latexsym,amsmath,amssymb,amscd,amsmath, mathrsfs, epsf, xypic, tikz-cd, enumitem}
\usepackage[left=25mm, right= 25mm, top=20mm, bottom=20mm, includefoot, includehead]{geometry}
\usepackage[textsize=tiny]{todonotes}
\usepackage{url}

\usepackage{tikz}
\usetikzlibrary{decorations.text,calc,arrows.meta}

\newtheorem{theorem}{Theorem}[section]
\newtheorem*{theo*}{Theorem}
\newtheorem{corollary}[theorem]{Corollary}
\newtheorem{lemma}[theorem]{Lemma}
\newtheorem{proposition}[theorem]{Proposition}
\newtheorem{conj}[theorem]{Conjecture}

\newtheorem{ex*}{Example}[section]
\theoremstyle{definition}
\newtheorem{definition}[theorem]{Definition}
\theoremstyle{remark}
\newtheorem{remark}[theorem]{Remark}
\theoremstyle{definition}

\newcommand{\brk}{{\underline{{\bf rk}}}}
\newcommand{\bb}[1]{\mathbf{#1}}

\newcommand{\CC}{\mathbb{C}}

\newcommand{\PP}{\mathbb{P}}

\newcommand{\NN}{\mathbb{N}}
\newcommand{\ZZ}{\mathbb{Z}}

\newcommand{\VSPb}{\underline{\mathrm{VSP}}}
\newcommand{\Ann}{\mathrm{Ann}}
\newcommand{\HF}{\mathrm{HF}}
\newcommand{\Hilb}{\mathrm{Hilb}}

\newcommand{\oldrho}{\iota_1^\#}

\makeatletter
\@namedef{subjclassname@2020}{%
	\textup{2020} Mathematics Subject Classification}
\makeatother

\date{}


\title{Symmetrization maps and minimal border rank Comon's conjecture}

\date{}

\author{Tomasz Ma\'ndziuk}
\address{Texas A\&M University, Department of Mathematics,
	College Station, TX 77843-3368, USA}
\email{t.mandziuk@tamu.edu}

\author{Emanuele Ventura}
\address{Politecnico di Torino, Dipartimento di Scienze Matematiche ``G. L. Lagrange'', Corso Duca degli Abruzzi 24\\
	10129 Torino, Italy}
\email{emanuele.ventura@polito.it}

\begin{document}
	
	\maketitle
	
	\begin{abstract}
		One of the fundamental open problems in the field of tensors is the {\it border Comon's conjecture}: given 
		a symmetric tensor $F\in(\CC^n)^{\otimes d}$ for $d\geq 3$,
		its border and symmetric border ranks are equal.
		In this paper, we prove the conjecture for large classes of concise tensors in $(\CC^n)^{\otimes d}$ of border rank $n$, i.e., tensors of minimal border rank. These families include all tame tensors and all tensors whenever $n\leq d+1$. 
		Our technical tools are border apolarity and border varieties of sums of powers.

	\end{abstract}
	
	\section{Introduction}
	Among the various significant notions of rank associated with tensors, border rank stands out for its geometric character, as it is defined in terms of secant varieties. Buczy\'nska and Buczy\'nski \cite{bb19} introduced a novel approach to lower bound this rank, {\it border apolarity}, which translates this geometric concept into an algebraic, ideal-theoretic framework.
	
	Border tensor rank arises in several important contexts, notably in computational complexity \cite{BCS97, Lands2011, Lands2017}, with roots tracing back to the work of Strassen \cite{Str69} and Bini \cite{Bin80}. As a core of border apolarity, Buczy\'nska and Buczy\'nski defined a parametrization of border rank decompositions via projective varieties, known as {\it border varieties of sums of powers} (denoted $\VSPb$). Recently, properties of these varieties have been explored further by the authors \cite{MV23}.
	
	In computational complexity, this development has provided a potential avenue for addressing the well-known {\it barriers} that hinder the use of vector bundle techniques in establishing strong lower bounds on border ranks \cite{Ga17, EGOW18}. These barriers are naturally explained in terms of schematic ranks, as the cactus rank is often significantly lower than the border rank (though not always), while vector bundle methods, such as flattening constructions, yield lower bounds for the former. The effectiveness of this new approach was demonstrated by Conner, Harper, and Landsberg \cite{chl19}, who derived the following lower bounds on the border rank of rectangular matrix multiplication: $\brk(M_{\langle 2,n,n\rangle}) \geq n^2 + 1.32n$ and $\brk(M_{\langle 3,n,n\rangle}) \geq n^2 + 1.6n$. Prior results lacked the linear terms in $n$.
	
	A fundamental open problem in the field of tensors, which we refer to as {\it border Comon's conjecture}, parallels Comon's conjecture for tensor rank \cite[\S 4.1]{CGLM08} \cite{GOV19}, exploring the tension between border rank and symmetric border rank \cite[\S 2, open problem (1)]{O08}. 
	
	This problem is largely open even for minimal border rank and constitutes the primary motivation for our paper. 
	
	\begin{conj}[{\bf Border Comon's conjecture for minimal border rank}\@]\label{quest}
		Let $F\in (\CC^n)^{\otimes d}$ be a concise symmetric tensor of minimal border tensor rank, i.e., $\brk(F)=n$, and let $p_F$ be its corresponding polynomial. Then $\brk_S(p_F) = \brk(F)$.  
	\end{conj}
	From definitions,  $\brk(F)\leq \brk_S(p_F)$. This conjecture has already been explicitly tackled e.g. in \cite{bgl13,Fri16}. 
	Specifically, Buczy\'nski, Ginensky, and Landsberg
	called it the {\it $\mathrm{BRPP}_n$ version of Comon's conjecture}, in our context restricted to the dense set of concise symmetric tensors. This was part of a general framework devised to formulate and approach several basic (and some of them yet unanswered to this day) questions on tensors in a unified terminology. In \cite{bgl13}, it is shown that Conjecture \ref{quest} holds true when the $n$th secant of the $d$th Veronese variety $\nu_d(\PP^{n-1})$
	is defined by flattenings
	(or, more generally, its set-theoretic equations are inherited from those of the corresponding secant of the Segre variety); however this condition does not cover the minimal border rank case for arbitrary choices of $n$ and $d$.
	For $1$-generic $3$-tensors ($d=3$), the validity of the border Comon's conjecture is known \cite[Proposition 5.6.1.6]{Lands2017}.

	\subsection{Results}
	We now showcase the content of this article. 
	The overall plan is to compare $\VSPb(F,r)$ and $\VSPb(p_F,r)$, the two varieties parametrizing border rank decompositions of a tensor $F$ and its corresponding homogeneous polynomial $p_F$ of degree $d$; see Definition \ref{defin:vspb}. 
	
	\begin{theo*}[Theorem \ref{theo:1}]
		Let $n\leq r\leq \binom{n+2}{2}$. There exists a  morphism (called \textnormal{desymmetrizing}) $\Upsilon: \VSPb(p_F,r)\rightarrow \VSPb(F,r)$. 
	\end{theo*}
	
	Given a product of projective space $\PP^{n_1}\times \cdots \times \PP^{n_d}$, its homogeneous coordinate ring (also called {\it Cox ring}) is a $\ZZ^d$-graded polynomial ring with degree $\mathbf{e}_i\in\ZZ^d$ variables corresponding to homogeneous coordinates of the $i$-th factor. Let $S=\CC[\alpha_{i,j} \mid 1\leq i \leq d$, $1\leq j\leq n]$ and $S^{sym}=\CC[\beta_j\mid 1\leq j\leq n]$ be the homogeneous coordinate rings of $X=(\mathbb{P}^{n-1})^{\times d}$ and $\mathbb{P}^{n-1}$, respectively. 
	In \S\ref{sec:pi} we introduce the ring map $\Delta^\#\colon S\longrightarrow S^{sym}$ given by $\alpha_{i,j}\mapsto \beta_j$ for all $i,j$, and an ideal $I_R = (\alpha_{i,j}\alpha_{k,\ell}-\alpha_{i,\ell}\alpha_{k,j} \mid 1\leq i < k \leq d, 1\leq j<\ell\leq n)\subset S$, somewhat encoding symmetrizations. Our main result is the following. 
	
	\begin{theo*}[Theorem \ref{thm:equality_of_br_in_terms_of_vspb}]
		Let $n\leq r\leq \binom{n+2}{2}$. Assume that a concise symmetric tensor  $F$ in $(\CC^n)^{\otimes d}$ has border rank $r$. The following conditions are equivalent:
		\begin{enumerate}
			\item[(i)] the corresponding homogeneous polynomial $p_F$ has symmetric border rank $r$;
			\item[(ii)] there exists $J \in \VSPb(F,r)$ such that $I_R\subseteq J$ and $\Delta^\#(J_{\mathbf{u}}) = \Delta^\#(J_{\mathbf{v}})$ for all $\mathbf{u},\mathbf{v}\in \ZZ^d$ with $|\mathbf{u}| = |\mathbf{v}|$;
			\item[(iii)] there exists $J \in \VSPb(F,r)$ such that $\Delta^\#(J_{(d,0,\ldots, 0)}) \subseteq \Delta^\#(J_{\mathbf{1}})$.
		\end{enumerate}
	\end{theo*}

 \noindent{\bf Outline of the proof of Theorem~\ref{thm:equality_of_br_in_terms_of_vspb}.}
    Since $\brk(F) = r$, it follows from border apolarity that there exists a $\mathbb{Z}^d$-homogeneous ideal in $S$ that is a point of $\VSPb(F,r)$. However, in general this ideal does not reflect the fact that $F$ is a symmetric tensor. We have $\brk(p_F) = r$ if and only if there exists a $\mathbb{Z}$-homogeneous ideal $I$ in $S^{sym}$ that is a point of $\VSPb(p_F,r)$. 

    Assume that such an $I$ exists, then using the desymmetrizing morphism $\Upsilon$ we construct an ideal $J=\Upsilon(I) \in \VSPb(F,r)$ with some special properties reflecting the fact that $J$ corresponds to a symmetric border rank decomposition.
    In fact, $\Delta^\#(J) = I$ so we have a straightforward way to recover from $J=\Upsilon(I)$ the original ideal $I$ that corresponds to a symmetric border rank decomposition.
    
    On the other hand, assume that $J$ satisfies weaker assumptions as in part (iii). We may construct an ideal $I' \in \VSPb(p_F, r)$ by the formula $I' = \iota_1^\#(J)$ ($\iota_1^\#$ is defined in Definition~\ref{def:rho}).

    If one wants to create a $\ZZ^d$-homogeneous ideal $J$ in $S$ from a $\ZZ$-homogeneous ideal $I$ in $S^{sym}$ in such a way that $J_{\mathbf{u}}$ ``corresponds'' to $I_{|\mathbf{u}|}$, then the information from $I_s$ has to be used to create all $J_{\mathbf{u}}$ with $|\mathbf{u}| = s$. If in addition, one does not want to distinguish any factor in $(\CC^{n})^{\otimes d}$, this creates strong compatibility conditions on $J$ as in part (ii).

    On the other hand, if $J\in \VSPb(F,r)$, then by ``projecting to factors'' one gets $d$ ideals in $S^{sym}$. These ideals are obtained by considering only the homogeneous parts of $J$ corresponding to lattice points on a fixed edge of the effective cone of $X$.
    In the process of passing from $J$ to its ``projection'' one loses a lot of information about $J$ so
    the assumptions that guarantee that the resulting ideal is in $\VSPb(p_F, r)$ are less restrictive than in part (ii). In fact, it always holds that these $d$ projected ideals are in $\mathrm{Slip}_{r,\mathbb{P}^{n-1}}$ (in particular they have the correct Hilbert function). However, in general it might happen that none of them is contained in $\Ann(p_F)$.
    Under the assumption of part (iii) the ``projection onto the first factor'' produces an ideal contained in $\Ann(p_F)$ and therefore, a point of $\VSPb(p_F, r)$. \\

	We believe this result may be employed to try to find a counterexample to border Comon's conjecture, not necessarily in minimal border rank.  
	The first corollary of Theorem~\ref{thm:equality_of_br_in_terms_of_vspb} improves on a result of Buczy\'nski, Ginensky, and Landsberg: they
	proved that whenever $F\in (\CC^n)^{\otimes d}$ has border rank $\brk(F)=r$ with $d\geq 2r-1$, then Conjecture \ref{quest} is true,
	without the concise condition on $F$. We improve this bound as follows.

	\begin{theo*}[Corollary \ref{cor:nleq d+1}]
		If a concise symmetric tensor $F\in (\CC^n)^{\otimes d}$ has border rank $n$ and $n\leq d+1$, then $p_F$ has minimal symmetric border rank.
	\end{theo*}

	The second result following from the main theorem establishes that Conjecture \ref{quest} holds for tame tensors, i.e., those whose smoothable rank is equal to border rank; see Definitions \ref{def:brk} and \ref{def:smoothrk}.  
	
	\begin{theo*}[Corollary \ref{cor:special_case_nonwil}]
		If a concise symmetric tame tensor $F\in (\CC^n)^{\otimes d}$ has minimal border rank, then $p_F$ has minimal symmetric smoothable rank. In particular, $p_F$ has minimal symmetric border rank.
	\end{theo*}
	
	In the context of $3$-tensors, Jelisiejew, Landsberg, and Pal introduced the important notion of $111$-sharpness \cite{JLP24}. We introduce the family
	of {\it sharp} tensors in $(\CC^{n})^{\otimes d}$ in Definition \ref{def:sharp}. For concise symmetric minimal border rank $3$-tensors, $111$-sharpness 
	is equivalent to sharpness; see Proposition \ref{prop:generalizes}. We show that Conjecture \ref{quest} is true for sharp tensors of minimal border rank, and so 
	 it is true for concise $111$-sharp $3$-tensors. 
	
	\begin{theo*}[Theorem \ref{thm:sharp}]
		If $d\geq 3$ and $F \in (\CC^n)^d$ is a symmetric sharp tensor of minimal border rank, then $\brk(F) = \brk_S(p_F)$.
	\end{theo*}
	
	\noindent {\bf Organization of the paper.} 
	In \S\ref{sec:prelim}, we collect notation and preliminaries. 
	In \S\ref{sec:pi}, we introduce 
	the map $\Delta^\#$ between the polynomial rings $S$ and $S^{sym}$ and record its properties. In \S\ref{sec: morphism between vspbars}, we construct the desymmetrizing morphism $\Upsilon$ and prove Theorem \ref{theo:1}. We introduce the {\it symmetrizing morphism} $\Sigma$ in Definition \ref{def:symmorphism}. Along with a locus inside a suitable flag multigraded Hilbert scheme, this morphism is used to obtain a sufficient condition providing the validity of Conjecture \ref{quest} for $d=3$; see Corollary \ref{prop:border_ranks_agree}. In \S\ref{sec:rho}, we introduce another map $\oldrho$ between $S$ and $S^{sym}$, and show our main Theorem \ref{thm:equality_of_br_in_terms_of_vspb} with its aforementioned corollaries. \\

	\noindent {\bf Acknowledgements.}
	We are grateful to J. Buczy\'nski, J. Jelisiejew and J.M. Landsberg for useful comments and discussions. 
	The first author was partially supported by NSF grant AF-2203618.
	The second author is a member of GNSAGA group of INdAM (Italy).   
	E.V. acknowledges that the preparation of this article was partially funded by the Italian Ministry of University and Research in the framework of the Call for Proposals for scrolling of final rankings of the PRIN 2022 call through the project {\it Applied Algebraic Geometry of Tensors}, Protocol no. 2022NBN7TL.
	 We thank an anonymous referee for very useful comments and corrections that also improved the presentation of the paper. 
	
	\section{Preliminaries}\label{sec:prelim}
	
	Let $S$ be the homogeneous coordinate ring of the $d$-factor Segre variety $X=(\PP^{n-1})^{\times d}$, that is,  $S = \CC[\alpha_{i,j}\mid 1\leq i\leq d$, $1\leq j\leq n]$, where $\mathrm{deg}(\alpha_{i,j})={\bf{e}}_i\in \ZZ^d$ is the $i$th standard basis vector. 
	Define  $\widetilde{S}=\CC[x_{i,j}\mid 1\leq i\leq d$, $1\leq j\leq n]$ to be the polynomial ring with the $\mathbb{Z}^d$-grading defined by $\mathrm{deg}(x_{i,j}) = \mathbf{e}_i$. We endow it with an $S$-module structure given by $\alpha_{i,j}\circ F = \frac{\partial F}{\partial x_{i,j}}$ for all $i,j$.
    Let $B_X\subset S$ be the irrelevant ideal of $S$, i.e.,  $B_X = (\alpha_{1,1},\ldots, \alpha_{1,n})\cdots (\alpha_{d,1},\ldots, \alpha_{d,n})$.
    Let ${\bf 1}$ denote  $(1,1,\ldots, 1)\in \ZZ^d$.
	
	The {\it multigraded Hilbert function} 
	of a homogeneous ideal $J\subset S$ is the numerical function $\mathrm{HF}(S/J,\cdot): \ZZ^d\to \NN$ defined by 
	\[
	\mathrm{HF}(S/J,\bb{v}) = \dim_{\CC} S_{\bb{v}} - \dim_{\CC} J_{\bb{v}}. 
	\]
	Given $0\neq F\in \widetilde{S}_{\bb{v}}$, denote $[F]\in \PP(\widetilde{S}_{\bb{v}})$ to be the corresponding point in projective space. 
	
	\begin{definition}[{\bf Apolar ideals}\@]
		Let $F\in \widetilde{S}_{\bb{v}}$. Then its {\it apolar} or {\it annihilator ideal} is the homogeneous ideal $\Ann(F) = \lbrace \psi\in S \ | \ \psi\circ F = 0\rbrace \subset S$.  
	\end{definition}
	
	\begin{proposition}[{\cite[proof of Theorem~1.4]{Ga23}}]\label{prop:inclusion in onedeg}
		If $F\in \widetilde{S}_{\bb v}$, then $I\subset \Ann(F) \Longleftrightarrow I_{\bb{v}}\subset \Ann(F)_{\bb{v}}$. 
	\end{proposition}
	
	\begin{definition}[{\bf Border rank}\@]\label{def:brk}
		For a point $[F]\in \PP(\widetilde{S}_{\bb{v}})$, the {\it border rank} of $F$ is the minimal integer $r\geq 1$ such that $[F]\in \sigma_r(X)$, the $r$-th secant variety of $X\subset \PP(\widetilde{S}_{\bb{v}})$. The border rank of $F$ is denoted  $\underline{{\bf rk}}(F)$.  
	\end{definition}
	
	\begin{definition}[{\bf Smoothable rank}\@]\label{def:smoothrk}
		The {\it smoothable rank} of $[F]\in \PP(\widetilde{S}_{\bb{v}})$ with respect to $X$ is the minimal integer $r\geq 1$ such that there exists a finite scheme $Z\subset X$ of length $r$ which is {\it smoothable} and $[F]\in \langle Z\rangle$. Equivalently, there exists a finite smoothable scheme $Z\subset X$ of length $r$ whose $B_X$-saturated ideal satisfies $I_Z\subset \mathrm{Ann}(F)\subset S$. The smoothable rank of $F$ is denoted  $\mathrm{srk}(F)$. 
	\end{definition}

	Smoothable and border ranks satisfy $\underline{{\bf rk}}(F)\leq \mathrm{srk}(F)$. Equality holds
	for general points of any secant variety.
	
	\begin{definition}[{\bf Wildness}\@]\label{def:wild}
		An element $F\in \widetilde{S}_{\bb{v}}$ is {\it wild} if $\mathrm{srk}(F) > \brk(F)$. When equality holds, we say that $F$ is {\it tame}. 
	\end{definition}
	\noindent Wildness was introduced in \cite{bb15},
	studied for forms in \cite{HMV} and for tensors in \cite{JLP24,MV23}.

	Given a numerical function $h: \ZZ^d\rightarrow \NN$, let $\mathrm{Hilb}^h_S$ be the scheme 
	whose closed points are the homogeneous ideals $I\subset S$ with $\mathrm{HF}(S/I,\cdot) = h$. This is the {\it Haiman-Sturmfels multigraded Hilbert scheme} \cite{hs}. 
	Since $S$ is positively graded, $\mathrm{Hilb}^h_S$ is a projective scheme for any Hilbert function $h$. 
	We ignore the scheme structure of $\mathrm{Hilb}^h_S$ and look at the underlying reduced scheme $(\mathrm{Hilb}^h_S)_{\mathrm{red}}$. 
	A closed point of $\mathrm{Hilb}^h_S$ corresponds to an ideal $I\subset S$ with Hilbert function $h$: we express this membership 
	in a set-theoretic fashion as $I\in \mathrm{Hilb}^h_S$. 
	
	For any integer $r\geq 0$, define the numerical function $h_{r,X}: \ZZ^d\rightarrow \NN$ to be 
	\[
	h_{r,X}(\bb{v}) = \min\lbrace r, \dim_{\CC} S_{\bb{v}}\rbrace. 
	\]
	The function $h_{r,X}$ is the {\it generic Hilbert function of $r$ points on $X$} \cite[\S 3.2]{bb19}. By \cite[Lemma~3.9]{bb19} the equality $\mathrm{HF}(S/I_Z,\cdot) = h_{r,X}$ holds for a very general collection of $r$ points $Z\subset X$. 
	
	Buczy\'nska and Buczy\'nski proved that, in the scheme $\mathrm{Hilb}^{h_{r,X}}_S$, the closure of the locus of all $I\in \mathrm{Hilb}^{h_{r,X}}_S$ that are the saturated ideals of $r$ distinct points in $X$ is an irreducible component \cite[Proposition 3.13]{bb19}, called  $\mathrm{Slip}_{r,X}$. Here the saturation is taken with respect to the irrelevant homogeneous ideal $B_X$ of $S$. We are now ready to state the border apolarity theorem \cite[Theorem 3.15]{bb19}.
	
	\begin{theorem}[{\bf Border apolarity}\@]\label{mainbb}
		For any $F\in \widetilde{S}_{{\bf v}}$, the following assertions are equivalent: 
		\begin{enumerate}
			
			\item[(i)] The border rank of $F$ with respect to $X$ satisfies $\brk(F)\leq r$;
			
			\item[(ii)] there exists a homogeneous ideal $I\subset \Ann(F)\subset S$ which lies in the irreducible component $\mathrm{Slip}_{r,X}$. 
		\end{enumerate}
	\end{theorem}
	
	This motivates the definition, originally introduced in \cite{bb19},  of the main tool of this article. 
	
	\begin{definition}[{\bf Varieties of sums of powers}\@]\label{defin:vspb}
		For $F\in \widetilde{S}_{{\bf v}}$, its $r$-th {\it variety of sums of powers} with respect to $X$ is the closed locus
		\[
		\VSPb(F,r) = \left\lbrace J\in \mathrm{Slip}_{r,X} \mbox{ such that } J\subset \mathrm{Ann}(F)\right\rbrace.
		\]
	\end{definition}

	Let $S^{sym} =\CC[\beta_j\mid 1\leq j\leq n]$ be the homogeneous coordinate ring of $\PP^{n-1}$ with standard grading and let $\widetilde{S^{sym}}=\CC[y_j\mid 1\leq j\leq n]$ be a standard $\mathbb{Z}$-graded polynomial ring with an $S^{sym}$-module structure defined by $\beta_{j}\circ p = \frac{\partial p}{\partial y_j}$ for all $j$. Theorem \ref{mainbb} applies to (and was primarily motivated by) tensors $F\in \widetilde{S}_{{\bf 1}}\cong (\CC^n)^{\otimes d}$. An $F\in (\CC^n)^{\otimes d}$ is a symmetric $d$-tensor whenever 
	$\mathfrak S_d\cdot F=F$, where $\mathfrak S_d$ acts naturally on $(\CC^n)^{\otimes d}$ by permuting  factors. Since we work in characteristic zero, the polynomial $p_F\in (\widetilde{S^{sym}})_d$ corresponding to the symmetric tensor $F$ is the image of $F$ under the equivariant projection map composed with the inverse of the polarization isomorphism, $\rho_d: (\CC^n)^{\otimes d}\longrightarrow S^d \CC^n\cong (\widetilde{S^{sym}})_d$. Their apolar ideals are inside the polynomial rings $S$ and $S^{sym}$ respectively, i.e., $\Ann(F)\subset S$ and $\Ann(p_F)\subset S^{sym}$.

	\begin{definition}[{\bf Symmetric border and smoothable ranks}\@]
		For a symmetric tensor $F\in (\CC^n)^{\otimes d}$, its {\it border rank} (resp. {\it smoothable rank}) is the one with respect to the Segre variety $X$. Its {\it symmetric  border rank} (resp. {\it symmetric smoothable rank}) is the one with respect to the Veronese variety $\nu_d(\PP^{n-1})$. The latter ones are denoted $\brk_S(p_F)$ and $\mathrm{srk}_S(p_F)$ respectively, where $p_F$ is the homogeneous polynomial of degree $d$ corresponding to $F$. 
	\end{definition}

	\begin{definition}[{\bf Conciseness}\@]
		A symmetric tensor $F\in (\CC^n)^{\otimes d}$ is {\it concise}
		when the induced linear map $\widetilde{F}:(\CC^{n})^{*}\longrightarrow (\CC^n)^{\otimes (d-1)}$
		is injective. A concise tensor $F\in (\CC^n)^{\otimes d}$ satisfies $\brk(F)\geq n$ and, when equality holds, $F$ is said to be of {\it minimal border rank}. 
	\end{definition}
	
	\subsection{Notation}
	Throughout the paper $F$ is a concise symmetric tensor in $(\CC^n)^{\otimes d}$ and $p_F$ is the corresponding degree $d$ homogeneous polynomial. Even though our main results concern the case where $F$ has minimal border rank, unless we explicitly state this assumption we consider the general situation. Recall that $X = (\PP^{n-1})^d$ has homogeneous coordinate ring $S = \CC[\alpha_{i,j}\mid 1\leq i\leq d$, $1\leq j\leq n]$ and $\widetilde{S}$ denotes $\CC[x_{i,j}\mid 1\leq i\leq d$, $1\leq j\leq n]$. Moreover, $\PP^{n-1}$ has homogeneous coordinate ring $S^{sym} =\CC[\beta_j\mid 1\leq j\leq n]$ and $\widetilde{S^{sym}}$ denotes $\CC[y_j\mid 1\leq j\leq n]$.

	\section{The diagonal morphism}\label{sec:pi}

	We define the following ring map between $S$ and $S^{sym}$. 
	
	\begin{definition}[{\bf Map} $\Delta^\#$\@]
		Let $\Delta^\#\colon S\longrightarrow S^{sym}$ be given by $\alpha_{i,j}\mapsto \beta_j$ for all $i,j$.
	\end{definition}
	
	Geometrically, it corresponds to the diagonal embedding $\Delta$ of affine spaces $\mathbb{A}^{n} \to (\mathbb{A}^n)^{\times d}$.

	\begin{lemma}\label{lem:number_of_sequences}
		Let $r$ and $n$ be positive integers. There are exactly $\binom{n+r-1}{r}$ non-decreasing integer sequences $1\leq a_1 \leq a_2 \leq \cdots \leq a_r \leq n$.
	\end{lemma}
	\begin{proof}
		Let $b_i = a_i + i-1$. The sequences $1\leq a_1 \leq a_2 \leq \cdots \leq a_r \leq n$ are in bijection with increasing sequences $1\leq b_1 < b_2 < \cdots < b_r \leq n + r-1$. The number of such sequences is $\binom{n+r-1}{r}$ corresponding to the choice of $r$ values that are taken by $b_1,b_2,\ldots, b_r$.
	\end{proof}

	\begin{lemma}\label{lem:image_of_degree_111}
		We have the inclusion $\Delta^\#(\Ann(F)_{\bf 1}) \subseteq \Ann(p_F)_d\subset S^{sym}$.
		\begin{proof}
			Let $\boldsymbol{\gamma}=(\gamma_1,\ldots,\gamma_n)\in \ZZ_{\geq 0}^n$
			with $|\boldsymbol{\gamma}|=\sum_i \gamma_i = d$, and let $\boldsymbol{y}^{\boldsymbol{\gamma}}=y_1^{\gamma_1}\cdots y_n^{\gamma_n}\in (\widetilde{S^{sym}})_d$ be a monomial of degree $|\boldsymbol{\gamma}|=d$. Furthermore, denote $\boldsymbol{\gamma}!=\gamma_1!\cdots \gamma_n!$. In the tensor space $(\CC^n)^d$ we fix a basis $x_{1,i_1}\otimes \cdots \otimes x_{d,i_d}$, where $1\leq i_k\leq n$. 
			Let $p_F = \sum_{\boldsymbol{\gamma}\colon |\boldsymbol{\gamma}|=d} \frac{d!}{\boldsymbol{\gamma}!} c^{\boldsymbol{\gamma}} {\boldsymbol{y}^{\boldsymbol{\gamma}}}
			\in (\widetilde{S^{sym}})_d$, where $c^{\boldsymbol{\gamma}}\in \CC$. Then the corresponding tensor is 
			\[
			F = \sum_{\boldsymbol{\gamma}\colon |\boldsymbol{\gamma}|=d}c^{\boldsymbol{\gamma}} \sum_{(i_1,\ldots, i_d)\in I_\gamma}x_{1,i_1}\otimes \cdots \otimes x_{d,i_d} \in (\CC^n)^{\otimes d},
			\]
			where $I_\gamma = \{(i_1, \ldots, i_d) \in \{1,\ldots, n\}^d\mid \#\{j \mid i_j = k\} = \gamma_k \text{ for all } 1\leq k \leq n\}$.
			Let $\theta = \sum_{\boldsymbol{\gamma}\colon |\boldsymbol{\gamma}| = d} \sum_{(i_1,\ldots, i_d) \in I_\gamma} A_{\boldsymbol{\gamma}, (i_1,\ldots, i_d)} \alpha_{1,i_1}\cdots \alpha_{d, i_d}
			\in \Ann(F)\subset S$. 
			Thus 
			\[
			\theta\circ F = \sum_{\boldsymbol{\gamma}\colon |\gamma| = d} \sum_{(i_1,\ldots, i_d)\in I_\gamma} c^{\boldsymbol{\gamma}}A_{\boldsymbol{\gamma}, (i_1,\ldots,i_d)} = 0.
			\]

			On the other hand, $\Delta^\#(\theta) = \sum_{\boldsymbol{\gamma}\colon |\gamma| = d} \sum_{(i_1,\ldots, i_d)\in I_\gamma} A_{\boldsymbol{\gamma}, (i_1,\ldots,i_d)} {\boldsymbol{\beta}}^{\boldsymbol{\gamma}}\in S^{sym}$.
			Now 
			\[
			\Delta^\#(\theta)\circ p_F =  \sum_{\boldsymbol{\gamma}\colon |\gamma| = d}\sum_{(i_1,\ldots, i_d)\in I_\gamma}  d!c^{\boldsymbol{\gamma}}A_{\boldsymbol{\gamma}, (i_1,\ldots,i_d)} = 0
			\]
			and the conclusion follows. 
		\end{proof}
	\end{lemma}
	
	\begin{definition}[{\bf Ideal} $I_R$]\label{def:I and I_R}
		Define
		\[
		I_R = (\alpha_{i,j}\alpha_{k,\ell}-\alpha_{i,\ell}\alpha_{k,j} \mid 1\leq i < k \leq d, 1\leq j<\ell\leq n)\subset S.
		\]
		
		As we verify computationally below, this is the ideal of the diagonal in $(\mathbb{A}^{n})^{\times d}$. Let $\mathcal A\subset \lbrace 0,1\rbrace^{\times d}$ be the subset of the $d$-tuples with only two $1$'s in their nonzero entries. Note that, for any concise symmetric tensor $F$, one has that
		$I_R\subseteq (\bigoplus_{{\bf v}\in \mathcal A}\Ann(F)_{{\bf v}})\subseteq \Ann(F)\subset S$. 
	\end{definition}
	
	\begin{lemma}\label{lem:ker_of_pi}
		For every $\mathbf{u}\in \ZZ^d$ we have $\ker \Delta^\#|_{S_\mathbf{u}} = (I_R)_{\mathbf{u}}$. Furthermore, if $J$ is a homogeneous ideal such that $J_{\mathbf{u}}$ contains $(I_R)_\mathbf{u}$ for some $\mathbf{u}\in \ZZ^d$, then 
		\[
		\dim_\CC (S^{sym})_{|\mathbf{u}|} / (\Delta^\#|_{S_\mathbf{u}}(J_\mathbf{u})) = \dim_\CC (S/J)_\mathbf{u}.
		\]
	\end{lemma}
	\begin{proof}
		By definition, $(I_R)_{\mathbf{u}} \subseteq \ker \Delta^\#|_{S_\mathbf{u}}$. Let $\theta \in S_\mathbf{u}$. We show that if $\Delta^\#(\theta) = 0$, then $\theta\in I_R$. By adding elements from $(I_R)_{\mathbf{u}}$ to $\theta$ we may assume that
		\begin{equation}\label{eq:expression_theta}
			\theta = \sum_{\substack{k_{1,1} \leq \cdots \leq k_{1,u_1}\leq \\ 
					\leq k_{2,1} \leq \cdots \leq k_{2,u_2}\leq \\ 
					\vdots \\ 
					\leq k_{d,1} \leq \cdots \leq k_{d,u_d}}} 
			\Gamma_{k_{1,1}\cdots k_{d,u_d}} \alpha_{1,k_{1,1}} \cdots \alpha_{1,k_{1,u_1}} \alpha_{2,k_{2,1}} \cdots \alpha_{2,k_{2,u_2}} \cdots \alpha_{d,k_{d,1}} \cdots \alpha_{d,k_{d,u_d}}
		\end{equation}
		for some $\Gamma_{k_{1,1}\cdots k_{d,u_d}}\in \CC$. If $\theta\in \ker \Delta^\#|_{S_\mathbf{u}}$, then
		\[
		0 = \Delta^\#(\theta) =\sum_{\substack{k_{1,1} \leq \cdots \leq k_{1,u_1}\leq \\ 
				\leq k_{2,1} \leq \cdots \leq k_{2,u_2}\leq \\ 
				\vdots \\ 
				\leq k_{d,1} \leq \cdots \leq k_{d,u_d}}} \Gamma_{k_{1,1}\cdots k_{d,u_d}}
		\beta_{k_{1,1}}\cdots \beta_{k_{d,u_d}}
		\]
		and hence $\Gamma_{k_{1,1}\cdots k_{d,u_d}}$
		are all zero. As a result $\theta = 0$ which proves the first statement of the lemma. One has the equalities 
		\[
		\dim_\CC (S^{sym})_{|\mathbf{u}|} / (\Delta^\#|_{S_\mathbf{u}}(J_\mathbf{u})) = \dim_\CC (S^{sym})_{|\mathbf{u}|} - (\dim_\CC J_\mathbf{u} - \dim_\CC (I_R)_\mathbf{u}) = 
		\]
		\[
		=\dim_\CC (S^{sym})_{|\mathbf{u}|} - \dim_\CC (S/I_R)_\mathbf{u}+ \dim_\CC (S/J)_\mathbf{u}.
		\]
		Then, it is sufficient to show that $\dim_\CC (S/I_R)_\mathbf{u} = \dim_\CC (S^{sym})_{|\mathbf{u}|}$. Observe that in expression \eqref{eq:expression_theta} of $\theta$ modulo $(I_R)_{\mathbf{u}}$ the coefficients $\Gamma_{k_{1,1}\cdots k_{d,u_d}}$ are uniquely determined. So the quotient space $S_{\mathbf{u}}/(I_R)_{\mathbf{u}}$ has a $\CC$-basis whose elements are in bijection with nondecreasing integer sequences $1\leq c_1\leq \cdots \leq c_{|\mathbf{u}|} \leq n$. By Lemma~\ref{lem:number_of_sequences} there are $\binom{n+|\mathbf{u}|-1}{|\mathbf{u}|}$ 
		such sequences, which is the dimension of $(S^{sym})_{|\mathbf{u}|}$. Hence $\dim_\CC (S^{sym})_{|\mathbf{u}|} = \dim_\CC (S/I_R)_\mathbf{u}$.
	\end{proof}

	\begin{lemma}\label{lem:image_is_equal}
		$\Delta^\#(\Ann(F)_{{\bf 1}}) = \Ann(p_F)_d$.
	\end{lemma}
	\begin{proof}
		By Lemma~\ref{lem:image_of_degree_111}, it is enough to show that $\dim_\CC (S^{sym})_d/\Delta^\#(\Ann(F)_{{\bf 1}}) = 1$. This follows from Lemma~\ref{lem:ker_of_pi} since $(I_R)_{{\bf 1}} \subseteq \Ann(F)_{{\bf 1}}$. \end{proof}

	\section{A morphism between border varieties of sums of powers}\label{sec: morphism between vspbars}
	
	Now we construct a morphism $\Hilb_{S^{sym}}^{h_{r,\PP^{n-1}}} \to \Hilb_{S}^{h_{r,X}}$ and show that, given a symmetric tensor $F \in (\CC^{n})^{\otimes d}$, it maps ideals in $\VSPb(p_F, r)$ to ideals in $\VSPb(F,r)$. For every $\mathbf{u} \in \ZZ^d$ we define a $\CC$-linear injective map $\psi_\mathbf{u}\colon (S^{sym})_{|\mathbf{u}|} \to S_{\mathbf{u}}$ by the following formula on the monomial basis
	\[
	\beta_{i_1}\beta_{i_2}\cdots \beta_{i_{|\mathbf{u}|}} \mapsto (\alpha_{1,i_1}\cdots \alpha_{1,i_{u_1}})(\alpha_{2,i_{u_1+1}}\cdots \alpha_{2,i_{u_1+u_2}})\cdots(\alpha_{d,i_{|{\bf u}|-u_d+1}}\cdots \alpha_{d,i_{|{\bf u}|}}),
	\]
	where $i_1 \leq i_2 \leq \cdots \leq i_{|\mathbf{u}|}$. Since $\Delta^\#$ is injective on the image of $\psi_{\mathbf{u}}$, it follows from Lemma~\ref{lem:ker_of_pi} applied with $J=I_R$ that
	\begin{equation}\label{eq:_direct_sum_decomposition}
		S_{\mathbf{u}} = (I_R)_{\mathbf{u}}\oplus \psi_{\mathbf{u}}((S^{sym})_{|\mathbf{u}|}).
	\end{equation}
	
	Let $\phi\colon \ZZ\to \ZZ$ be any Hilbert function. Given an ideal $I\in \Hilb_{S^{sym}}^{\phi}$, let 
	\[
	\Upsilon(I) := I_R + \bigoplus_{\mathbf{u}\in \ZZ^d} \psi_{\mathbf{u}}(I_{|\mathbf{u}|})\subset S.
	\]
    Observe that for every $\mathbf{u}\in \mathbb{Z}^d$ we have $\Delta^\#(\Upsilon(I)_\mathbf{u}) = I_{|\mathbf{u}|}$. In particular, $\Upsilon(I) = (\Delta^\#)^{-1}(I)$.
	Let $\phi'\colon \ZZ^d\to \ZZ$ be defined by $\phi'(\mathbf{u}) = \phi(|\mathbf{u}|)$. 
	
	\begin{proposition}\label{lem:morphism_well_defined}
		If $I\in \Hilb_{S^{sym}}^{\phi}$, then $\Upsilon(I) \in \Hilb_S^{\phi'}$.
	\end{proposition}
	\begin{proof}
		By construction $\Upsilon(I)$ is a $\ZZ^d$-graded $\CC$-vector subspace of $S$.
		We first check that it is an ideal. Since $I_R$ is an ideal it is enough to show that if $\overline{h} = \psi_{\mathbf{u}}(h)$ for some $h\in I_{|\mathbf{u}|}$, then $\alpha_{i,j} \overline{h}$  are all in $\Upsilon(I)$ for every $1\leq i\leq d$ and $1\leq j\leq n$. This follows from the fact that $\psi_{\mathbf{u}+\mathbf{e}_1}(\beta_j h) \equiv \alpha_{1,j}\overline{h} \mod (I_R)_{\mathbf{u}+\mathbf{e}_1}$ and similar equalities for other $\alpha_{i,j}$. 
		
		By \eqref{eq:_direct_sum_decomposition} we have
		\[
		\frac{S_{\mathbf{u}}}{(I_R)_\mathbf{u} + \psi_{\mathbf{u}}(I_{|\mathbf{u}|})}\cong \frac{\psi_\mathbf{u}((S^{sym})_{|\mathbf{u}|})}{\psi_{\mathbf{u}}(I_{|\mathbf{u}|})} \cong \frac{(S^{sym})_{|\mathbf{u}|}}{I_{|\mathbf{u}|}}.
		\]
		So $S/\Upsilon(I)$ has Hilbert function $\phi'$.
	\end{proof}
	
	Based on Proposition~\ref{lem:morphism_well_defined} we make the following definition.
	
	\begin{definition}[{\bf Desymmetrizing morphism}\@]
		The {\it desymmetrizing morphism} is the morphism of schemes $\Upsilon\colon \Hilb_{S^{sym}}^{\phi} \to \Hilb_S^{\phi'}$ defined on closed points by $I\mapsto \Upsilon(I)$.
	\end{definition}
	
	\begin{remark}
		It is interesting to consider the problem of functoriality of multigraded Hilbert schemes corresponding to smooth projective toric varieties.
		If $f$ is a good enough morphism between such varieties, then it induces a map of multigraded Hilbert schemes as described in \cite[Theorem~2.6]{Man23}. The morphism $f$ can be expressed in terms of a homomorphism of Cox rings \cite[Theorem~3.2]{Cox95}. One may ask  what  conditions on such a ring homomorphism guarantee that it induces a morphism of multigraded Hilbert schemes. The desymmetrizing morphism does {\it not} come from a homomorphism of Cox rings.
		What are other algebraic or geometric constructions that produce  morphisms mapping limits of radical and saturated ideals to limits of radical and saturated ideals? All these questions can also be asked in a more general setup, when we consider rational maps of varieties and we replace toric varieties by Mori dream spaces. Here for the algebraic description of the map instead of \cite{Cox95} one can use \cite{BB13} and \cite{BK18}. 
	\end{remark}
	
	We are mainly interested in the case where $\phi=h_{r,\PP^{n-1}}$ and $\phi'=h_{r,X}$ and we eventually restrict to this case. However, we start with some general observations.
	
	\begin{proposition}\label{lem:rad_to_rad}
		If $I$ is a radical ideal, then so is $\Upsilon(I)$.
	\end{proposition}
	\begin{proof}
	Assume that $f^k\in S$ is in $\Upsilon(I)$ for some positive integer $k$. Since $\Upsilon(I) = (\Delta^\#)^{-1}(I)$ we get $(\Delta^\#(f))^k = \Delta^\#(f^k) \in I$.
    Since $I$ is radical we get $\Delta^\#(f) \in I$, i.e., $f\in \Upsilon(I)$.
    \end{proof}
	
	\begin{proposition}\label{lem:sat_to_sat}
		If $I$ is a $B_{\PP^{n-1}}$-saturated ideal, then $\Upsilon(I)$ is a $B_X$-saturated ideal.
	\end{proposition}
	\begin{proof}
		Assume that $f\in S$ satisfies $f\in \Upsilon(I)\colon B_X$. Since $\Upsilon(I)$ and $B_X$ are $\mathbb{Z}^d$-homogeneous, it follows that for all $\mathbf{u}\in \mathbb{Z}^d$ the degree $\mathbf{u}$ part $f_\mathbf{u}$ of $f$ satisfies $f_\mathbf{u} \in \Upsilon(I)\colon B_X$. To conclude that $f\in \Upsilon(I)$, it is enough to show that each $f_\mathbf{u}$ is in $\Upsilon(I)$. Therefore, we may assume that $f$ is homogeneous.
        
        Assume that $f\in S_{\mathbf{u}}$ satisfies $f\alpha_{1,i_1}\cdots \alpha_{d,i_d}\in \Upsilon(I)$ for all $1\leq i_j \leq n$. By \eqref{eq:_direct_sum_decomposition} we can write $f$ in a unique way as $f=g+\psi_{\mathbf{u}}(h)$ where $g\in (I_R)_{\mathbf{u}}$ and $h\in (S^{sym})_{|\mathbf{u}|}$.
		
		Let $\overline{h} = \psi_{\mathbf{u}}(h)$. 
		From $f\alpha_{1,i_1}\cdots \alpha_{d,i_d}\in \Upsilon(I)$ we conclude that $ \overline{h}\alpha_{1,i_1}\cdots \alpha_{d,i_d}\in \Upsilon(I)$. Observe that the class in $(S/I_R)_{\mathbf{u}+\mathbf{1}}$ of $\overline{h}\alpha_{1,i_1}\cdots \alpha_{d,i_d}$ is the same as the class of $\psi_{\mathbf{u}+\mathbf{1}}(h \beta_{i_1}\cdots \beta_{i_d})$. Therefore, we obtain $h\beta_{i_1}\cdots \beta_{i_d}\in I$ for all $i_j$. Since $I$ is $B_{\PP^{n-1}}$-saturated we derive $h\in I$ and so $\overline{h}\in \Upsilon(I)$. This yields $f=g+\overline{h}\in \Upsilon(I)$.
	\end{proof}
	
	\begin{proposition}\label{lem:apolar_to_apolar}
		If $I\subseteq \Ann(p_F)$, then $\Upsilon(I)\subseteq \Ann(F)$.  
	\end{proposition}
	\begin{proof}
		By Proposition~\ref{prop:inclusion in onedeg} it is enough to show that if $f\in \Upsilon(I)_{\mathbf{1}}$, then $f\in \Ann(F)$. By \eqref{eq:_direct_sum_decomposition} we can write $f$ in a unique way as $f=g+\psi_{\mathbf{1}}(h)$ where $g\in (I_R)_{\mathbf{1}}$ and $h\in I_{d}$. Let $\overline{h} = \psi_{\mathbf{1}}(h)$.
		Since $F$ is symmetric, $I_R\subseteq \Ann(F)$ so it is enough to show that $\overline{h}\in \Ann(F)$.
		We have $h \in I_d \subseteq \Ann(p_F)$. Then, by Lemma~\ref{lem:image_is_equal}, there exists $h'\in \Ann(F)_{\mathbf{1}}$ with $\Delta^\#(h') = h$. Since $(I_R)_{\mathbf{1}} = \ker \Delta^\#|_{S_{\mathbf{1}}}$, we can assume using decomposition \eqref{eq:_direct_sum_decomposition} that $h' \in \psi_{\mathbf{1}}((S^{sym})_d)$. But $\Delta^\#|_{\psi_{\mathbf{1}}((S^{sym})_d)}$ is injective and $\overline{h}$ is an element of $\psi_{\mathbf{1}}((S^{sym})_d)$ with $\Delta^\#(\overline{h}) = h$. Hence, $h' = \overline{h}$ and we conclude that $\overline{h}\in \Ann(F)$.
	\end{proof}
	
	\begin{theorem}\label{theo:1}
		If $n\leq r \leq \binom{n+1}{2}$, then the desymmetrizing morphism $\Upsilon$ defines a morphism of multigraded Hilbert schemes $\Upsilon:\Hilb_{S^{sym}}^{h_{r,\PP^{n-1}}}\to \Hilb_S^{h_{r,X}}$ such that
		\begin{enumerate}
			\item[(i)] if $I\in \mathrm{Slip}_{r,\PP^{n-1}}$, then $\Upsilon(I)\in \mathrm{Slip}_{r,X}$;
			\item[(ii)] if $I\subseteq \Ann(p_F)$, then $\Upsilon(I)\subseteq \Ann(F)$.
		\end{enumerate}
		In particular, it restricts to a morphism from $\VSPb(p_F, r)$ to $\VSPb(F, r)$.
	\end{theorem}
	\begin{proof}
		We first prove that the Hilbert function of $S/\Upsilon(I)$ is $h_{r,X}$. Since $r\geq n = \dim_\CC (S^{sym})_1$, it follows that $I_1 =0$ and so, $\Upsilon(I)_\mathbf{u} = 0$ for every $\mathbf{u}\in \ZZ^d_{\geq 0}$ with $|\mathbf{u}| \leq 1$. As a result, for such $\mathbf{u}$, one has $\dim_\CC (S/\Upsilon(I))_\mathbf{u} = \dim_\CC S_\mathbf{u} = \min \{r, \dim_\CC S_\mathbf{u}\} = h_{r,X}(\mathbf{u})$. Assume that $|\mathbf{u}| \geq 2$. Then $h_{r,\PP^{n-1}}(|\mathbf{u}|) = \min \{\dim_\CC (S^{sym})_{|\mathbf{u}|}, r\} = r$ by the assumption that $r\leq \binom{n+1}{2} = \dim_\CC (S^{sym})_{2}$. 
		Then, $\dim_\CC (S/\Upsilon(I))_\mathbf{u} = r$ by Proposition~\ref{lem:morphism_well_defined}.
		On the other hand, $\dim_\CC S_{\mathbf{u}} \geq \binom{n+1}{2}$ so 
		\[
		\dim_\CC (S/\Upsilon(I))_{\mathbf{u}} = r = \min\{r, \dim_\CC S_\mathbf{u}\} = h_{r,X}(\mathbf{u}).
		\]
		
		It follows from Proposition~\ref{lem:morphism_well_defined} that $I\mapsto \Upsilon(I)$ defines a morphism of multigraded Hilbert schemes. Statement (i) follows from Propositions~\ref{lem:rad_to_rad} and \ref{lem:sat_to_sat}; statement (ii) follows from Proposition~\ref{lem:apolar_to_apolar}.
	\end{proof}

	We continue to assume that $n\leq r \leq \binom{n+1}{2}$.
	Let $g$ be the Hilbert function of $S/I_R$ and let $\Hilb_S^{h_{r,X}, g}$ be the flag multigraded Hilbert scheme equipped with its projections 
    \[\pi_1\colon \Hilb_S^{h_{r,X}, g}\to \Hilb_S^{h_{r,X}} \text{ and } \pi_2\colon \Hilb_S^{h_{r,X}, g}\to \Hilb_S^{g}.
    \]

	\begin{definition}[{\bf Symmetrizing morphism}\@]\label{def:symmorphism}
		We define a morphism $\xi\colon \pi_2^{-1}(I_R) \to \Hilb_{S^{sym}}^{h_{r,\PP^{n-1}}}$ on closed points by sending $I_R \subseteq J$ to $\Sigma(J)$ where 
		$\Sigma$ is the {\it symmetrizing morphism}. To define this latter morphism, we need to introduce more notation. Let $\mathcal E\subset \lbrace 0,1\rbrace^{\times d}$ be the set of $d$-tuples $\mathbf{0}, \mathbf{e}_1, {\bf{e}}_1+{\bf{e}}_2, \cdots, {\bf{e}}_1+\ldots+{\bf{e}}_{d-1}$. Define $a{\bf 1}=(a,a,\ldots,a)$. Then, on closed points, $\Sigma$ is defined by
		\[
		\Sigma(J) :=\bigoplus_{a\in \ZZ} \Big(\bigoplus_{{\bf s}\in \mathcal E} \Delta^\#(J_{a{\bf 1}+{\bf s}})\Big).
		\]
	\end{definition}
	Note that, for each $a\in \ZZ$, the summands have degrees $da, da+1,\ldots, d(a+1)-1$. So these are organized as residue classes modulo $d$. 
	
	We check that $\xi$ is well-defined: by the definition of $\Delta^\#$ the graded $\CC$-vector space $\Sigma(J)$ is an ideal. Its Hilbert function is $h_{r,\PP^{n-1}}$ by Lemma~\ref{lem:ker_of_pi}.
	
	Next we claim that if $J \in \Hilb_S^{h_{r,X}}$ is a radical and $B_X$-saturated ideal containing $I_R$, then $\xi(I_R\subseteq J)=\Sigma(J)$ is radical. In particular, $\Sigma(J)$ is in $\mathrm{Slip}_{r, \mathbb{P}^{n-1}}$.
	
	\begin{proposition}\label{lem:radical_if _radical}
		Assume that $I_R\subseteq J$ with $J\in \Hilb_S^{h_{r,X}}$. If $J$ is radical and $B_X$-saturated, then $\Sigma(J)$ is a radical ideal of $S^{sym}$. 
	\end{proposition}
	\begin{proof}
		Let $g\in (S^{sym})_N$ satisfy $g^k \in \Sigma(J)$ for some positive integer $k$. We can write $N = ad + m$ where $0\leq m < d$. Let $\mathbf{s} \in \{0,1\}^{\times d}$ be the element with $s_i = 1$ if and only if $i\leq m$. There exists $G\in S_{a\mathbf{1} + \mathbf{s}}$ with $g=\Delta^\#(G)$. To prove the statement it is sufficient to verify that $G \in J$, because then $g\in \Sigma(J)$.
		
		For every tuple $\mathbf{j}  = (j_{m+1},\ldots, j_{d})$ of integers between $1$ and $n$, the polynomial
        $(\beta_{j_{m+1}}\cdots \beta_{j_d}g)^k$ is in $\Sigma(J)$. Hence, there exists a polynomial $P(\mathbf{j}) \in J_{k(a+1)\mathbf{1}}$ with $\Delta^\#(P(\mathbf{j})) = (\beta_{j_{m+1}}\cdots \beta_{j_d}g)^k$.
		From the equality $\Delta^\#((\prod_{i=m+1}^d \alpha_{i,j_{i}} G)^k) = (\beta_{j_{m+1}}\cdots \beta_{j_{d}} g)^k$ it follows by Lemma~\ref{lem:ker_of_pi} that $(\prod_{i=m+1}^d \alpha_{i,j_{i}} G)^k - P(\mathbf{j}) \in I_R$. By assumption, $I_R\subseteq J$ and by construction $P(\mathbf{j})\in J$. Hence $(\prod_{i=m+1}^d \alpha_{i,j_{i}} G)^k \in J$.
		From the assumption that $J$ is radical we get $\prod_{i=m+1}^d \alpha_{i,j_{i}} G\in J$. Since $J$ is $B_X$-saturated, it is  $(\alpha_{m+1,1}, \ldots, \alpha_{m+1,n})\cdots (\alpha_{d,1},\ldots, \alpha_{d,n})$-saturated. So $G\in J$. 
	\end{proof}

	\begin{proposition}\label{lem:in_Slip_if_in_Slip}
		Let $n\leq r \leq \binom{n+1}{2}$. Suppose every irreducible component of $\pi_2^{-1}(I_R)$ contains a radical and $B_X$-saturated ideal. If $J\in \mathrm{Slip}_{r,X}$ contains $I_R$, then $\Sigma(J)$ is in $\mathrm{Slip}_{r,\PP^{n-1}}$. 
	\end{proposition}
	\begin{proof}
		If $J\in \mathrm{Slip}_{r,X}$ is radical and $B_X$-saturated, then $\Sigma(J)$ is in $\mathrm{Slip}_{r,\PP^{n-1}}$ by Proposition ~\ref{lem:radical_if _radical}. 
		Let $Z\subset \pi_2^{-1}(I_R)$ be any irreducible component. Consider the restriction of the first projection $\pi_{1|Z}\colon Z\longrightarrow \Hilb_S^{h_{r,X}}$. 
		
		Let $\mathrm{Sip}_{r,X}$ be the set of radical and $B_X$-saturated ideals and $Z^\circ = \pi_{1|Z}^{-1}(\mathrm{Sip}_{r,X})$.
		We claim that $\overline{Z^\circ} = \pi_{1|Z}^{-1}(\mathrm{Slip}_{r,X})$. Since $\pi_{1|Z}$ is injective, it is enough to prove that the images of these two sets under $\pi_{1|Z}$ are equal. To simplify notation let $\phi$ denote $\pi_{1|Z}$.
		Since $\phi$ is a closed map one has
		\[
		\phi(\overline{Z^{\circ}}) = \overline{\phi(Z^\circ)} = \overline{\mathrm{Sip}_{r,X}\cap \phi(Z)} = \phi(Z).
		\]
		Here the third equality is proven as follows: since $\phi(Z)$ is irreducible and $\mathrm{Sip}_{r,X}$ is open, the condition $\mathrm{Sip}_{r,X}\cap \phi(Z)\neq \emptyset$ implies that the latter is nonempty open and so dense in $\phi(Z)$. Moreover
		$\phi(\phi^{-1}(\mathrm{Slip}_{r,X})) = \phi(Z)$. The statement follows since $\xi$ is a closed map.
	\end{proof}

	\begin{corollary}
		Keep the assumption from Proposition \ref{lem:in_Slip_if_in_Slip}. If there exists $J\in \VSPb(F,r)$ containing $I_R$, then $K:=\Upsilon(\Sigma(J))\in \VSPb(F,r)$. In particular, the ideal $K$ is such that $\Delta^\#(K_{\mathbf{a}}) = \Delta^\#(K_{\mathbf{b}})$ for all multidegrees $\mathbf{a}, \mathbf{b}$ with $|\mathbf{a}| = |\mathbf{b}|$.
	\end{corollary}
	
	\begin{corollary}\label{prop:border_ranks_agree}
		Keep the assumption from Proposition \ref{lem:in_Slip_if_in_Slip}. Let $F$ be a concise symmetric tensor in $(\CC^n)^{\otimes d}$ of border tensor rank $r$. If there exists $J\in \VSPb(F,r)$ containing $I_R$, then the corresponding homogeneous polynomial $p_F$ of degree $d$ has  symmetric border rank $r$. In particular, if $F$ is a concise symmetric tensor in $(\CC^n)^{\otimes 3}$ of minimal border rank, then the corresponding polynomial $p_F$ has minimal  symmetric border rank. 
	\end{corollary}
	\begin{proof}
		By Proposition~\ref{lem:in_Slip_if_in_Slip}, the ideal $\Sigma(J)\in \mathrm{Slip}_{r,\PP^{n-1}}$ and since $J_{{\bf 1}}\subseteq \Ann(F)_{{\bf 1}}$ it follows from Lemma~\ref{lem:image_of_degree_111} that
		$\Sigma(J)_d = \Delta^\#(J_{{\bf 1}}) \subseteq \Ann(p_F)_d$. Hence, by Proposition \ref{prop:inclusion in onedeg}, $\Sigma(J)\subseteq \Ann(p_F)$. As a result $\Sigma(J) \in \VSPb(p_F, r)$. Then, by Theorem~\ref{mainbb}, the symmetric border rank of $p_F$ is at most $r$ and so equal to $r$. 
		
		Assume that $F$ is a concise symmetric tensor of minimal border rank. By Theorem~\ref{mainbb} there exists an ideal $J\in \VSPb(F,n)$. As $d=3$, by comparing Hilbert functions, we have $J_{{\bf v}} = \Ann(F)_{{\bf v}}$ for all ${\bf v}\in \mathcal A$ (see Definition~\ref{def:I and I_R}). Since $F$ is symmetric $I_R\subseteq \Ann(F)$ and we conclude that $I_R\subseteq J$. Now we may apply the first part of the statement. 
	\end{proof}

	So Corollary \ref{prop:border_ranks_agree} states that
	if every irreducible component of $\pi_2^{-1}(I_R)$ contains a radical and $B_X$-saturated ideal, then the minimal border rank Comon's Conjecture \ref{quest} holds for concise tensors in $(\CC^n)^{\otimes 3}$. However, we do not know whether this is an effective condition to check, for a description of $\pi_2^{-1}(I_R)$ might be hard to achieve. 
	Though it may be potentially ineffective, we believe this was an interesting construction to record.

	\section{A morphism in the other direction}\label{sec:rho}
	
	\begin{definition}[{\bf Map} $\oldrho$\@]\label{def:rho}
		Let $\oldrho\colon S\to S^{sym}$ be given by $\alpha_{1,j} \mapsto \beta_j$ for all $1\leq j \leq n$ and $\alpha_{i,j}\mapsto 0$ for all $2\leq i \leq d$ and all $1\leq j \leq n$. 
	\end{definition}

    Geometrically, $\iota_1^\#$ corresponds to the inclusion $\iota_1\colon \mathbb{A}^{n} \to (\mathbb{A}^{n})^{\times d}$ on the first factor.
    
	\begin{proposition}[{\cite[Corollary~2.7]{Man23}}]\label{prop:map_of_Slips}
		Let $Z=\PP^{n_1}\times \cdots \times \PP^{n_k}$ and $W = Z \times \PP^{m_1}\times \cdots \times \PP^{m_l}$. Then, the homogeneous coordinate ring $C[Z]$ of $Z$ is a subring of the homogeneous coordinate ring $C[W]$ of $W$ and for every positive integer $r$ there is a morphism of multigraded Hilbert schemes $\Hilb_{C[W]}^{h_{r, W}} \to \Hilb_{C[Z]}^{h_{r,Z}}$ given on closed points by $J\mapsto J\cap C[Z]$.
		Furthermore, this morphism maps $\mathrm{Slip}_{r, W}$ onto $\mathrm{Slip}_{r,Z}$.
	\end{proposition}
	
	In our case $C[\PP^{n-1}]=S^{sym}\subseteq C[X]=S$ can be realized by $\beta_j\mapsto \alpha_{1,j}$, and then $J\cap C[\PP^{n-1}]=J\cap S^{sym}$ becomes $\oldrho(J)$. 
	
	We use $\oldrho$ as well as $\Upsilon$ to characterize concise symmetric tensors for which symmetric border rank is equal to border rank in terms of existence of ideals with special properties in their border varieties of sums of powers.
	
	\begin{theorem}\label{thm:equality_of_br_in_terms_of_vspb}
		Let $n\leq r \leq \binom{n+1}{2}$. Assume that a concise symmetric tensor  $F$ in $(\CC^n)^{\otimes d}$ has border rank $r$. The following conditions are equivalent:
		\begin{enumerate}
			\item[(i)] the corresponding homogeneous polynomial $p_F$ has symmetric border rank $r$;
			\item[(ii)] there exists $J \in \VSPb(F,r)$ such that $I_R\subseteq J$ and $\Delta^\#(J_{\mathbf{u}}) = \Delta^\#(J_{\mathbf{v}})$ for all $\mathbf{u},\mathbf{v}\in \ZZ^d$ with $|\mathbf{u}| = |\mathbf{v}|$;
			\item[(iii)] there exists $J \in \VSPb(F,r)$ such that $\Delta^\#(J_{(d,0,\ldots, 0)}) \subseteq \Delta^\#(J_{\mathbf{1}})$.
		\end{enumerate}
		Furthermore, if these conditions hold, then $\oldrho(J) \in \VSPb(p_F, r)$.
	\end{theorem}
	\begin{proof}
		Assume that $(i)$ holds. By Theorem~\ref{mainbb} there exists $I\in \VSPb(p_F,r)$. Let $J = \Upsilon(I)$. By the definition of $\Upsilon$ we have $I_R\subseteq J$ and $\Delta^\#(J_\mathbf{u}) = I_{|\mathbf{u}|} = I_{|\mathbf{v}|} = \Delta^\#(J_{\mathbf{v}})$ for all $\mathbf{u}, \mathbf{v}\in \mathbb{Z}^d$ with $|\mathbf{u}| = |\mathbf{v}|$. Furthermore, by Theorem~\ref{theo:1}, $J$ is in $\VSPb(F, r)$. Therefore, $J$ is as in $(ii)$.
        
		The implication $(ii) \Rightarrow (iii)$ is clear. Assume that $(iii)$ holds. 
		By Proposition~\ref{prop:map_of_Slips} the ideal $I:=\oldrho(J)$ is in $\mathrm{Slip}_{r, \PP^{n-1}}$. Furthermore, 
		\[
		I_d = \oldrho(J_{(d,0,\ldots, 0)}) = \Delta^\#(J_{(d,0,\ldots, 0)}) \subseteq  \Delta^\#(J_{\mathbf{1}}) \subseteq \Delta^\#(\Ann(F)_{\mathbf{1}}) = \Ann(p_F)_d
		\]
		where the last equality follows from Lemma~\ref{lem:image_is_equal}. It follows that $I\in \VSPb(p_F, r)$ so by Theorem~\ref{mainbb} the symmetric border rank of $p_F$ is at most $r$ and thus it is equal to $r$.
	\end{proof}
	
	Next we exhibit three situations in which statement $(iii)$ of Theorem~\ref{thm:equality_of_br_in_terms_of_vspb} is satisfied.
	
	\begin{corollary}\label{cor:nleq d+1}
		If a concise symmetric tensor $F\in (\CC^n)^{\otimes d}$ has border rank $n$ and $n\leq d+1$, then $p_F$ has minimal symmetric border rank.
	\end{corollary}
	\begin{proof}
		Let $J\in \VSPb(F,n)$. By Theorem~\ref{thm:equality_of_br_in_terms_of_vspb} it is enough to show that $\Delta^\#(J_{(d,0,\ldots, 0)}) \subseteq \Delta^\#(J_{\mathbf{1}})$.
		Assume by contradiction that $\Delta^\#(J_{(d,0,\ldots, 0)}) \not\subseteq \Delta^\#(J_{\mathbf{1}}) $. Let $K\subset S^{sym}$ be the ideal generated by 
		$\Delta^\#(J_{\mathbf{1}})$, $\Delta^\#(J_{(d,0,\ldots, 0)})$ and $\Delta^\#(J_{(d,1,1,\ldots, 1)})$. The Hilbert function of $S^{sym}/K$ in degree $d$ is at most $n-1 \leq d$. Macaulay's bound \cite[Theorem~4.2.10]{BH98} states the Hilbert function $\HF(S^{sym}/K, a)$ is bounded from the above by $\HF(S^{sym}/K,a-1)^{\langle a-1\rangle}$ for all $a\geq d$,
		where the latter is the $(a-1)$-th Macaulay representation of the integer $\HF(S^{sym}/K,a-1)$. Note that $(n-1)^{\langle d\rangle}=n-1$, as $n-1\leq d$. 
		Thus $\HF(S^{sym}/K, a)$ is bounded from the above by $n-1$ for all $a\geq d$. 

       Since $F$ is symmetric we have $I_R \subseteq \Ann(F)$. From the assumption that $F$ is concise one gets
        \[
        \HF(\mathbf{1}-\mathbf{e}_i, S/\Ann(F)) = \HF(\mathbf{e}_i, S/\Ann(F)) = n = \HF(\mathbf{1}-\mathbf{e}_i, S/J)
        \]
        for every $i$.
        Therefore, $J_{\mathbf{1}-\mathbf{e}_i} = \mathrm{Ann}(F)_{\mathbf{1}-\mathbf{e}_i}$ for every $i$. In particular, 
        \begin{equation}\label{eq:5_4}
        (I_R)_{\mathbf{1}-\mathbf{e}_i} \subseteq J_{\mathbf{1}-\mathbf{e}_i} \text{ for all }i.
        \end{equation} We claim that $(I_R)_{(d,1,\ldots, 1)} \subseteq J_{(d,1,\ldots, 1)}$. Indeed, every element of $(I_R)_{(d,1,\ldots, 1)}$ can be written as a sum of elements of the form $HMM'$ where $M, M'$ are monomials, $H\in I_R$ and one of the following holds for some $i\neq j$ between $2$ and $d$.
        \begin{itemize}
            \item $H\in (I_R)_{\mathbf{e}_1 + \mathbf{e}_i}$, $M\in S_{\mathbf{1}-\mathbf{e}_1-\mathbf{e}_i-\mathbf{e}_j}$, $M''\in S_{(d-1)\mathbf{e}_1 +\mathbf{e}_j}$
            \item $H\in (I_R)_{\mathbf{e}_i + \mathbf{e}_j}$, $M\in S_{\mathbf{1}-\mathbf{e}_1-\mathbf{e}_i-\mathbf{e}_j}$, $M'\in S_{d\mathbf{e}_1}$.
        \end{itemize}
         In either case, it follows from \eqref{eq:5_4} that $HM\in J$. Therefore, $HMM'\in J$ and so their sum is also in $J$.
        By Lemma~\ref{lem:ker_of_pi}, the codimension of $\Delta^\#(J_{(d,1,\ldots, 1)})$ in $(S^{sym})_{2d-1}$ is equal to 
		$n$. Note that $K_{2d-1} = \Delta^\#(J_{(d,1,\ldots, 1)})$ by construction. Therefore, by the above, $\HF(S^{sym}/K,2d-1) = n$. This contradiction finishes the proof.
	\end{proof}

	\begin{corollary}\label{cor:special_case_nonwil}
		If a concise symmetric tame tensor $F\in (\CC^n)^{\otimes d}$ has minimal border rank, then $p_F$ has minimal symmetric smoothable rank. In particular, $p_F$ has minimal symmetric border rank.
	\end{corollary}
	\begin{proof}
		Let $J$ be a $B_X$-saturated ideal of a length $n$ smoothable scheme apolar to $F$. From $J\subseteq \Ann(F)$ one concludes that $S/J$ has Hilbert function $h_{n,X}$.
        This condition implies that $(I_R)_{\mathbf{1}-\mathbf{e}_i} \subseteq \Ann(F)_{\mathbf{1}-\mathbf{e}_i} = J_{\mathbf{1}-\mathbf{e}_i}$ for all $i$. Similarly to the argument in the last part of the proof of Corollary \ref{cor:nleq d+1}, one checks that $(I_R)_\mathbf{1} \subseteq J_\mathbf{1}$.
		Since the natural map $\mathrm{Slip}_{n, X}\to \mathrm{Hilb}_{sm}^n(X)$ is surjective (see \cite[Proposition~3.17]{MV23}), we conclude that $J\in \mathrm{Slip}_{n,X}$.
		
		By Proposition~\ref{prop:map_of_Slips}, $\oldrho(J)$ is in $\mathrm{Slip}_{n, \PP^{n-1}}$ and it is in fact saturated (see \cite[Proposition~2.1]{Man23} for details).
		So it is enough to show that $\oldrho(J)\subseteq \Ann(p_F)$ to conclude that the smoothable rank of $p_F$ satisfies $\mathrm{srk}_S(p_F)\leq n$, and thus is equal to $n$.
		By Theorem~\ref{thm:equality_of_br_in_terms_of_vspb}, to demonstrate this claim, it is enough to show that $\Delta^\#(J_{(d,0,\ldots, 0)}) \subseteq \Delta^\#(J_{\mathbf{1}})$.
		
		Since $(I_R)_{(d,0,\ldots, 0)} = 0$ and $(I_R)_\mathbf{1} \subseteq J_\mathbf{1}$, by Lemma~\ref{lem:ker_of_pi} both $\Delta^\#(J_{(d,0,\ldots, 0)})$ and $\Delta^\#(J_{\mathbf{1}})$ are of codimension $n$ in $(S^{sym})_d$. It is enough to show that if $F\in J_{\mathbf{1}}$, then $\Delta^\#(F) \in \Delta^\#(J_{(d,0,\ldots, 0)})$.
		
		Let $\tau\colon S^{sym}\hookrightarrow S$ be the injective ring map $\beta_j\mapsto \alpha_{1,j}$ for all $1\leq j\leq n$ and let $f=\tau(\Delta^\#(F))$. To derive the above conclusion, it is enough to verify that $f\in J$: this would imply that $\Delta^\#(F) = \Delta^\#(f) \in  \Delta^\#(J_{(d,0,\ldots,0)})$.
		
		For every $1\leq j_1,\ldots, j_d\leq n$, $\Delta^\#(F\alpha_{1,j_1}\cdots \alpha_{1,j_d}) = \Delta^\#(f \alpha_{1,j_1}\cdots \alpha_{d,j_d})$. From Lemma~\ref{lem:ker_of_pi}, we see that  
		$F\alpha_{1,j_1}\cdots \alpha_{1,j_d} - f \alpha_{1,j_1}\cdots \alpha_{d,j_d} \in (I_R)_{(d+1,1,\ldots, 1)} \subseteq J_{(d+1,1,\ldots, 1)}$ and hence, $f\alpha_{1,j_1}\cdots \alpha_{d,j_d}$ in $J$.
		Since $J$ is $B_X$-saturated we conclude that $f\in J$.
	\end{proof}
	
	\begin{remark}
		Let $F$ be a concise tame $3$-tensor of minimal border rank. Then $F$ is symmetric and the equality $\brk(F)=\brk_S(p_F)=n$ can be proven using results of Jelisiejew, Landsberg and Pal \cite{JLP24}. By \cite[Theorem 9.2]{JLP24}, since $F$ has minimal smoothable rank, then $F$ is $1$-generic, $111$-sharp and its $111$-algebra is smoothable and Gorenstein. Moreover, they prove that its $111$-algebra is such that its spectrum lays on the third Veronese variety and spans $p_F$: this implies that the symmetric smoothable rank of $p_F$ is minimal and the desired conclusion follows. Note that when $F$ is a $1$-generic $3$-tensor, the validity of the border Comon's conjecture was already known \cite[Proposition 5.6.1.6]{Lands2017}.
	\end{remark}

	The last family of tensors for which part $(iii)$ of Theorem~\ref{thm:equality_of_br_in_terms_of_vspb} is satisfied is a very interesting one. 
	We introduce the family of {\it sharp} $d$-tensors generalizing (see Proposition~\ref{prop:generalizes} below) the one of $111$-sharp concise $3$-tensors, introduced in  \cite{JLP24}, when we look at symmetric minimal border rank tensors.
	
	\begin{definition}[{\bf $111$-sharpness}\@]\label{def:111sharp}
		A concise tensor $F\in(\CC^n)^{\otimes 3}$ is \emph{$111$-sharp} if its annihilator $\Ann(F)$ has exactly $n-1$ minimal generators of multidegree $(1,1,1)$.
	\end{definition}
	
	A concise minimal border rank tame $3$-tensor is $111$-sharp \cite[Theorem~9.2]{JLP24} so this family generalizes the one considered in Corollary~\ref{cor:special_case_nonwil} when $d=3$. 
	
	\begin{definition}[{\bf Sharp tensors}\@]\label{def:sharp}
		Let $F\in (\CC^n)^{\otimes d}$ with $d\geq 3$. We say that $F$ is {\it sharp} if 
		\begin{enumerate}
			\item\label{it:wild_1} $\Ann(F)$ has $n-1$ minimal generators of degree $\mathbf{1}$;
			\item\label{it:wild_2} $\mathrm{HF} (S/\Ann(F), \mathbf{u}) = n$ for all $0 < \mathbf{u} < \mathbf{1}$;
			\item\label{it:wild_3} $\mathrm{HF} (S/(\Ann(F)_{\mathbf{e}_i + \mathbf{e}_j}), s\mathbf{e}_i + \mathbf{e}_j) = n$ for all $1\leq i \neq j \leq n$ and all $1\leq s \leq d-1$.
		\end{enumerate}
	\end{definition}

	We need the following corollary of Lemma~\ref{lem:ker_of_pi}.

	\begin{corollary}\label{cor:image_of_degree_1-e_d}
		Let $F$ be a symmetric concise tensor in $(\CC^n)^{\otimes d}$ for some $d\geq 3$.
		$\Delta^\#(\Ann(F)_{\mathbf{1} - \mathbf{e}_d}) = \Ann(p_F)_{d-1}$. 
	\end{corollary}
	\begin{proof}
		Since $F$ is concise, so is $p_F$. So $\dim_\CC (S^{sym}/\Ann(p_F))_{d-1} = n = \dim_\CC (S/\Ann(F))_{\mathbf{1}-\mathbf{e}_d}$.
		Since $I_R \subseteq \Ann(F)$, by Lemma~\ref{lem:ker_of_pi} both of the subspaces of $(S^{sym})_{d-1}$ in the statement are of codimension $n$.
		Therefore, it is sufficient to check that $\Delta^\#(\Ann(F)_{\mathbf{1} - \mathbf{e}_d}) \subseteq \Ann(p_F)_{d-1}$.
		
		If $\theta\in \Ann(F)_{\mathbf{1}-\mathbf{e}_d}$, then $\alpha_{d, j}\theta\in\Ann(F)_{\mathbf{1}}$ for all $j\in \{1,\ldots,n\}$. 
		By Lemma~\ref{lem:image_is_equal}, for every $j\in \{1,2,\ldots, n\}$ one has
		$\beta_j\Delta^\#(\theta) \in \Ann(p_F)_d$. We conclude that $\Delta^\#(\theta) \in \Ann(p_F)_{d-1}$ and hence $\Delta^\#(\Ann(F)_{\mathbf{1}-\mathbf{e}_d}) \subseteq \Ann(p_F)_{d-1}$.
	\end{proof}

	\begin{lemma}\label{lem:also_sharp_general}
		Let $F$ be a symmetric concise tensor in $(\CC^n)^{\otimes d}$. The number of minimal generators of $\Ann(p_F)$ of degree $d$ is the same as the number of minimal generators of $\Ann(F)$ of degree $\mathbf{1}$.
	\end{lemma}
	\begin{proof}
		Let $J = \sum_{i=1}^d (\Ann(F)_{\mathbf{1}-\mathbf{e}_i})$. If $\Ann(F)$ has $s$ minimal generators of degree $\mathbf{1}$, then $\dim_\CC (S/J)_{\mathbf{1}} = s + 1$.
		So $\dim_\CC (S^{sym})_d/\Delta^\#(J_{\mathbf{1}}) = s + 1$ by Lemma~\ref{lem:ker_of_pi}. 
		
		Since $F$ is symmetric, the image of $J_{\mathbf{1}-\mathbf{e}_i}$ under $\Delta^\#$ is independent of $i$.
		Then, we have $\Delta^\#(J_{\mathbf{1}}) = (\Delta^\#(J_{\mathbf{1}-\mathbf{e}_d}))_d$.
		By Corollary~\ref{cor:image_of_degree_1-e_d} we have $\Delta^\#(J_{\mathbf{1}-\mathbf{e}_d}) = \Delta^\#(\Ann(F)_{\mathbf{1}-\mathbf{e}_d}) =  \Ann(p_F)_{d-1}$. From $\dim_\CC (S^{sym})_d/ \Delta^\#(J_{\mathbf{1}}) = s + 1$ we conclude that $\Ann(p_F)$ has $s$ minimal homogeneous generators of degree $d$. 
	\end{proof}
	
	\begin{lemma}\label{lem:containment}
		Let $F$ be a symmetric concise tensor in $(\CC^n)^{\otimes d}$. If $I = \sum_{0 < \mathbf{u} < \mathbf{1}} (\Ann(F)_\mathbf{u})$, then $\Delta^\#(I_{(d-1)\mathbf{e}_1+\mathbf{e}_2}) \subseteq \Ann(p_F)_{d}$.
	\end{lemma}
	\begin{proof}
		Consider two families of claims depending on $s$:
		\begin{align}
			\Delta^\#(I_{s\mathbf{e}_1 + \mathbf{e}_2 + \cdots + \mathbf{e}_{d-s}}) &\subseteq \Ann(p_F)_{d-1} \tag*{A(s)}\\
			\Delta^\#(I_{(s+1)\mathbf{e}_1 + \mathbf{e}_2 + \cdots + \mathbf{e}_{d-s}}) &\subseteq (\Ann(p_F)_{d-1})_d \tag*{B(s)}.
		\end{align}
		$A(1)$ holds by Corollary~\ref{cor:image_of_degree_1-e_d} and our goal is to prove that $B(d-2)$ holds. If $s\geq 1$, then $I_{(s+1)\mathbf{e}_1 + \mathbf{e}_2 + \cdots + \mathbf{e}_{d-s}} = S_{\mathbf{e}_1} \cdot I_{s\mathbf{e}_1 + \mathbf{e}_2 + \cdots + \mathbf{e}_{d-s}}$, because there are no new generators in multidegree $(s+1)\mathbf{e}_1 + \mathbf{e}_2 + \cdots + \mathbf{e}_{d-s}$. Consequently, $B(s)$ follows from $A(s)$ for all $1\leq s \leq d-2$. 
		
		Assume that $1\leq s \leq d-3$ and $B(s)$ holds. Let $\theta \in I_{(s+1)\mathbf{e}_1 + \mathbf{e}_2 + \cdots + \mathbf{e}_{d-s-1}}$. Then, for every $1\leq j \leq n$, $\alpha_{d-s,j}\cdot\theta \in I_{(s+1)\mathbf{e}_1 + \mathbf{e}_2 + \cdots + \mathbf{e}_{d-s}}$, so by $B(s)$ we have $\beta_{j}\cdot \Delta^\#(\theta) \in \Ann(p_F)_d$ for every $j$.
		It follows that $\Delta^\#(\theta) \in \Ann(p_F)_{d-1}$, i.e., $A(s+1)$ holds.
	\end{proof}

	\begin{proposition}\label{prop:generalizes}
		If $d=3$, a concise symmetric tensor is sharp if and only if it is $111$-sharp as in Definition \ref{def:111sharp}.
	\end{proposition}
	\begin{proof}
		For $d=3$, every concise tensor $F$ satisfies condition~\eqref{it:wild_2} in Definition~\ref{def:sharp} and $F$ is $111$-sharp if and only if it satisfies condition
		\eqref{it:wild_1}. So it is sufficient to prove that, for $d=3$, condition \eqref{it:wild_1} implies condition \eqref{it:wild_3}.
		
		By Lemma~\ref{lem:ker_of_pi}, in order to show that  $\mathrm{HF} (S/(\Ann(F)_{\mathbf{e}_i + \mathbf{e}_j}), 2\mathbf{e}_i + \mathbf{e}_j) = n$, it is enough to show that 
		$\dim_\CC (S^{sym})_3/\Delta^\#((\Ann(F)_{\mathbf{e}_i + \mathbf{e}_j})_{2\mathbf{e}_i + \mathbf{e}_j}) = n$. From Corollary~\ref{cor:image_of_degree_1-e_d} one gets 
		$\Delta^\#(\Ann(F)_{\mathbf{e}_i + \mathbf{e}_j}) = \Ann(p_F)_2$ and hence,
		\[
		\dim_\CC (S^{sym})_3/\Delta^\#((\Ann(F)_{\mathbf{e}_i + \mathbf{e}_j})_{2\mathbf{e}_i + \mathbf{e}_j}) = \dim_\CC (S^{sym})_3/(\Ann(p_F)_2)_3 = n,
		\]
		by the assumption that $F$ is $111$-sharp together with Lemma~\ref{lem:also_sharp_general}.
	\end{proof}
	
	\begin{theorem}\label{thm:sharp}
		If $d\geq 3$ and $F\in (\CC^n)^{\otimes d}$ is a concise symmetric sharp tensor of minimal border rank, then the homomorphism $\oldrho: S\longrightarrow S^{sym}$
		induces a morphism of reduced schemes $(\oldrho)_{\VSPb}: \VSPb(F,n)\longrightarrow \VSPb(p_F,n)$. In particular, $\brk(F) = \brk_S(p_F)$.
	\end{theorem}

	\begin{proof}
		By Proposition~\ref{prop:map_of_Slips} the ring homomorphism $\oldrho$ induces a morphism of multigraded Hilbert schemes that restricts to a morphism $\mathrm{Slip}_{n,X}\to \mathrm{Slip}_{n,\PP^{n-1}}$.
		
		Let $K\in \VSPb(F,n)$. In order, to show that the above morphism restricts a morphism $(\oldrho)_{\VSPb}: \VSPb(F,n)\longrightarrow \VSPb(p_F,n)$ it is enough to show that $\oldrho(K) \subseteq \Ann(p_F)$. Let $I = \sum_{0 < \mathbf{u} < \mathbf{1}} (\Ann(F)_\mathbf{u})$. Since $F$ is sharp, we have $I \subseteq K$. We claim that $K_{(d-1)\mathbf{e}_1 + \mathbf{e}_2} = I_{(d-1)\mathbf{e}_1 + \mathbf{e}_2}$.
		Indeed, since $F$ is sharp, we have $\dim_\CC (S/I)_{(d-1)\mathbf{e}_1 + \mathbf{e}_2} = n = \dim_\CC (S/K)_{(d-1)\mathbf{e}_1 + \mathbf{e}_2}$, where the
		first equality follows from property \eqref{it:wild_3} of Definition \ref{def:sharp}. 
		
		We show that $\oldrho(K_{(d-1)\mathbf{e}_1})  = \Ann(p_F)_{d-1}$. Since $p_F$ is concise, it is enough to show that $\oldrho(K_{(d-1)\mathbf{e}_1}) \subseteq \Ann(p_F)_{d-1}$. Furthermore, since $p_F$ is of degree $d$, it is sufficient to show that $(S^{sym})_1\cdot \oldrho(K_{(d-1)\mathbf{e}_1}) \subseteq \Ann(p_F)_{d}$. Let $\theta \in K_{(d-1)\mathbf{e}_1}$ and $j\in \{1,\ldots, n\}$. Then $\alpha_{2,j}\cdot \theta \in K_{(d-1)\mathbf{e}_1 + \mathbf{e}_2} = I_{(d-1)\mathbf{e}_1 + \mathbf{e}_2}$. It follows from Lemma~\ref{lem:containment} that 
		\[
		\beta_j\cdot \oldrho(\theta) = \Delta^\#(\alpha_{2,j}\cdot \theta) \in \Ann(p_F)_d,
		\]
		for any $j\in \{1,\ldots, n\}$.
		Since $p_F$ has degree $d$, then $\oldrho(\theta)\in \Ann(p_F)_{d-1}$, so $\oldrho(K)_{d-1} \subseteq \Ann(p_F)_{d-1}$. These vector spaces have the same dimension, so they are equal. 
		Thus, 
		\[
		\oldrho(K)_d \supseteq (\oldrho(K)_{d-1})_d = (\Ann(p_F)_{d-1})_d. 
		\]
		Since $F$ is sharp, $\Ann(p_F)$ has $n-1$ minimal generators in degree $d$ by Lemma~\ref{lem:also_sharp_general}. It follows that the rightmost vector space is of codimension $n$ in $(S^{sym})_d$.
		By the equality of dimensions, $\oldrho(K)_d = (\Ann(p_F)_{d-1})_d$
		and in particular, $\oldrho(K) \subseteq \Ann(p_F)$.
		
		Since $F$ has minimal border rank $\VSPb(F,n)$ is nonempty by Theorem~\ref{mainbb}. Therefore, so is $\VSPb(p_F, n)$. Using Theorem~\ref{mainbb} we conclude that $\brk_S(p_F) = n$.
	\end{proof}
	Corollary \ref{cor:nleq d+1} says that any minimal border rank tensor $F\in (\CC^n)^{\otimes 3}$ whenever $n\leq 4$ satisfies border Comon's conjecture. However, we can do a little bit better using a result from \cite{JLP24}.  
	
	\begin{proposition}\label{cor:n leq 5}
		Any minimal border rank symmetric concise tensor $F\in (\CC^n)^{\otimes 3}$ for $n\leq 5$ satisfies $\brk(F)=\brk_S(p_F)$.
		\begin{proof}
			If $n\leq 5$ then every minimal border rank tensor is $111$-sharp \cite[Subsection~1.4.1]{JLP24}. So Theorem \ref{thm:sharp} and Proposition~\ref{prop:generalizes} establish the statement.
		\end{proof}
	\end{proposition}

	Theorem \ref{thm:sharp} also slightly strengthens a result on cubic forms by Seigal \cite[Corollary~1.6]{Sei20}. 
	
	\begin{corollary}
		If $\brk_S(p_F)\leq 6$, then $\brk(F) = \brk_S(p_F)$.
	\end{corollary}
	\begin{proof}
		By \cite[Corollary~1.6]{Sei20} we may assume that $\brk_S(p_F) = 6$. If $F$ is a concise tensor in $(\CC^6)^{\otimes 3}$, the conclusion is clear. If $F$ is a concise tensor of minimal border rank in $(\CC^5)^{\otimes 3}$, then Proposition \ref{cor:n leq 5} implies $\brk(F) = \brk_S(p_F) = 5$. Thus, we may assume that, up to the action of the general linear group, $p_F$ is a cubic polynomial in four variables. In this case \cite[Theorem~1.5]{Sei20} implies $\brk_S(p_F) = \brk(F)$.
	\end{proof}

	\bibliographystyle{amsplain}
	
	\begin{small}
		
	\end{small}
\end{document}